\documentclass{article}%
\usepackage{amssymb}
\usepackage{amsmath}
\usepackage{amsfonts}
\usepackage{graphicx}%
\providecommand{\U}[1]{\protect\rule{.1in}{.1in}}
\newtheorem{theorem}{Theorem}[section]

\newtheorem{corollary}[theorem]{Corollary}

\newtheorem{lemma}[theorem]{Lemma}

\newtheorem{problem}[theorem]{Problem}
\newtheorem{proposition}[theorem]{Proposition}
\newtheorem{remark}[theorem]{Remark}

\newenvironment{proof}[1][Proof]{\noindent\textbf{#1.} }{\ \rule{0.5em}{0.5em}}
\begin{document}

\author{Vadim E. Levit\\Department of Mathematics\\Ariel University, Israel\\levitv@ariel.ac.il
\and Eugen Mandrescu\\Department of Computer Science\\Holon Institute of Technology, Israel\\eugen\_m@hit.ac.il}
\title{On the Number of Vertices/Edges whose Deletion Preserves the
K\"{o}nig-Egerv\'{a}ry Property}
\date{}
\maketitle

\begin{abstract}
Let $\alpha(G)$ denote the cardinality of a maximum independent set, while
$\mu(G)$ be the size of a maximum matching in $G=\left(  V,E\right)  $. If
$\alpha(G)+\mu(G)=\left\vert V\right\vert $, then $G$ is a
\textit{K\"{o}nig-Egerv\'{a}ry graph}.

The number $d\left(  G\right)  =\max\{\left\vert A\right\vert -\left\vert
N\left(  A\right)  \right\vert :A\subseteq V\}$ is the \textit{critical
difference} of the graph $G$, where $N\left(  A\right)  =\left\{  v:v\in
V,N\left(  v\right)  \cap A\neq\emptyset\right\}  $. Every set $B\subseteq V$
satisfying $d\left(  G\right)  =\left\vert B\right\vert -\left\vert N\left(
B\right)  \right\vert $ is \textit{critical}. Let $\varepsilon\left(
G\right)  =\left\vert \mathrm{\ker}(G)\right\vert $ and $\xi\left(  G\right)
=\left\vert \emph{core}\left(  G\right)  \right\vert $, where $\mathrm{\ker
}(G)$ is the intersection of all critical independent sets, and \emph{core}%
$\left(  G\right)  $ is the intersection of all maximum independent sets. It
is known that $\mathrm{\ker}(G)\subseteq$ \emph{core}$\left(  G\right)  $
holds for every graph.

Let $\varrho_{v}\left(  G\right)  =\left\vert \left\{  v\in V:G-v\text{ is a
K\"{o}nig-Egerv\'{a}ry graph}\right\}  \right\vert $ and

$\varrho_{e}\left(  G\right)  =\left\vert \left\{  e\in E:G-e\text{ is a
K\"{o}nig-Egerv\'{a}ry graph}\right\}  \right\vert $. Clearly, $\varrho
_{v}\left(  G\right)  =\left\vert V\right\vert $ and $\varrho_{e}\left(
G\right)  =\left\vert E\right\vert $ for bipartite graphs. Unlike the
bipartiteness, the property of being a K\"{o}nig-Egerv\'{a}ry graph is not hereditary.

In this paper, we show that%
\[
\varrho_{v}\left(  G\right)  =\left\vert V\right\vert -\xi\left(  G\right)
+\varepsilon\left(  G\right)  \text{ and }\varrho_{e}\left(  G\right)
\geq\left\vert E\right\vert -\xi\left(  G\right)  +\varepsilon\left(
G\right)
\]
for every K\"{o}nig-Egerv\'{a}ry graph $G$.


\end{abstract}

\section{Introduction}

Throughout this paper $G=(V,E)$ is a finite, undirected, loopless graph
without multiple edges, with vertex set $V=V(G)$ of cardinality $\left\vert
V\left(  G\right)  \right\vert =n\left(  G\right)  $, and edge set $E=E(G)$ of
size $\left\vert E\left(  G\right)  \right\vert =m\left(  G\right)  $.

If $X\subset V$, then $G[X]$ is the subgraph of $G$ induced by $X$. By $G-X$
we mean the subgraph $G[V-X]$, for $X\subseteq V\left(  G\right)  $, and we
write $G-v$ instead of $G-\left\{  v\right\}  $. If $A,B$ $\subseteq V\left(
G\right)  $ and $A\cap B=\emptyset$, then $(A,B)$ stands for the set
$\{e=ab:a\in A,b\in B,e\in E\left(  G\right)  \}$. The \textit{neighborhood}
of a vertex $v\in V\left(  G\right)  $ is the set $N(v)=\{w:w\in V$ and $vw\in
E\}$. The \textit{neighborhood} of $A\subseteq V$ is $N(A)=\{v\in V:N(v)\cap
A\neq\emptyset\}$. If $\left\vert N(A)\right\vert >\left\vert A\right\vert $
for every independent set $A$, then $G$ is a \textit{regularizable
graph}\emph{ }\cite{Berge1982}\emph{.}

A set $S\subseteq V$ is \textit{independent} if no two vertices belonging to
$S$ are adjacent. Let $\mathrm{Ind}(G)$ denote the family off all independent
sets of $G$. The \textit{independence number} $\alpha(G)$ is the size of a
largest independent set (i.e., of a \textit{maximum independent set}) of $G$.
Let $\Omega(G)=\{S:S$ \textit{is a maximum independent set of} $G\}$ and
$\mathrm{core}(G)=\cap\{S:S\in\Omega(G)\}$, while $\xi(G)=\left\vert
\mathrm{core}(G)\right\vert $ \cite{LevMan2002a}. A vertex $v\in V(G)$ is
$\alpha$-\textit{critical }provided $\alpha(G-v)<\alpha(G)$. Clearly,
$\mathrm{core}(G)$ is the set of all $\alpha$-critical vertices of $G$. An
edge $e\in E(G)$ is $\alpha$-\textit{critical }provided $\alpha(G)<\alpha
(G-e)$. Let $\eta\left(  G\right)  $ denote the number of $\alpha$-critical
edges of graph $G$ \cite{LevMan2006}. Notice that there are graphs in which
every edge is $\alpha$-critical (e.g., all $C_{2k+1}$ for $k\geq1$) or no edge
is $\alpha$-critical (e.g., all $C_{2k}$ for $k\geq2$).

A \textit{matching} in a graph $G=(V,E)$ is a set of edges $M\subseteq E$ such
that no two edges of $M$ share a common vertex. A matching of maximum
cardinality $\mu(G)$ is a \textit{maximum matching}, and a \textit{perfect
matching} is one saturating all vertices of $G$. Given a matching $M$ in $G$,
a vertex $v\in V\left(  G\right)  $ is $M$-\textit{saturated} if there exists
an edge $e\in M$ incident with $v$. An edge $e\in E(G)$ is $\mu$%
-\textit{critical }provided $\mu(G-e)<\mu(G)$. A vertex $v\in V(G)$ is $\mu
$-\textit{critical }provided $\mu(G-v)<\mu(G)$, i.e., $v$ is $M$-saturated by
every maximum matching $M$ of $G$.

It is known that
\[
\left\lfloor \frac{n\left(  G\right)  }{2}\right\rfloor +1\leq\alpha
(G)+\mu(G)\leq n\left(  G\right)  \leq\alpha(G)+2\mu(G)
\]
hold for every graph $G$ \cite{BGL2002}. If $\alpha(G)+\mu(G)=n\left(
G\right)  $, then $G$ is called a K\"{o}nig-Egerv\'{a}ry graph\textit{
}\cite{dem,ster}, while if $\alpha(G)+\mu(G)=n\left(  G\right)  -1$, then $G$
is called a $1$-K\"{o}nig-Egerv\'{a}ry graph \cite{LevMan2023}. For instance,
both $G_{1}$ and $G_{2}$ from Figure \ref{Fig123} are $1$%
-K\"{o}nig-Egerv\'{a}ry graphs.\begin{figure}[h]
\setlength{\unitlength}{1cm}\begin{picture}(5,1.2)\thicklines
\multiput(2,0)(1,0){4}{\circle*{0.29}}
\multiput(2,1)(1,0){3}{\circle*{0.29}}
\put(2,0){\line(1,0){3}}
\put(2,0){\line(0,1){1}}
\put(2,0){\line(1,1){1}}
\put(2,0){\line(2,1){2}}
\put(2,1){\line(1,-1){1}}
\put(2,1){\line(2,-1){2}}
\put(2,1){\line(3,-1){3}}
\put(3,0){\line(0,1){1}}
\put(3,0){\line(1,1){1}}
\put(3,1){\line(1,-1){1}}
\put(3,1){\line(2,-1){2}}
\put(4,0){\line(0,1){1}}
\put(4,1){\line(1,-1){1}}
\put(5.4,0){\makebox(0,0){$v_{1}$}}
\put(4.4,1){\makebox(0,0){$v_{2}$}}
\put(1,0.5){\makebox(0,0){$G_{1}$}}
\multiput(8,0)(1,0){5}{\circle*{0.29}}
\multiput(11,1)(1,0){2}{\circle*{0.29}}
\put(8,1){\circle*{0.29}}
\put(8,0){\line(1,0){4}}
\put(8,0){\line(0,1){1}}
\put(8,1){\line(1,-1){1}}
\put(11,1){\line(1,0){1}}
\put(10,0){\line(1,1){1}}
\put(12,0){\line(0,1){1}}
\put(9.5,0.2){\makebox(0,0){$e_{1}$}}
\put(10.55,0.2){\makebox(0,0){$e_{2}$}}
\put(7,0.5){\makebox(0,0){$G_{2}$}}
\end{picture}\caption{$G_{1}-v_{1}$, $G_{2}-e_{2}$ are K\"{o}nig-Egerv\'{a}ry
graphs, while $G_{1}-v_{2}$ and $G_{2}-e_{1}$ are not K\"{o}nig-Egerv\'{a}ry
graphs.}%
\label{Fig123}%
\end{figure}
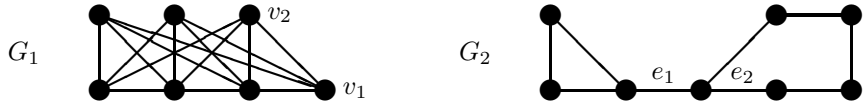

A graph is: \textit{(i)} \textit{unicyclic} if it has a unique cycle
\cite{LevMan2012b}, and \textit{(ii)} \textit{almost bipartite} if it has only
one odd

cycle \cite{LevMan2022}.

\begin{lemma}
\label{lem84}\cite{LevMan2022} If $G$ is an almost bipartite graph, then
$n(G)-1\leq\alpha(G)+\mu(G)\leq n(G)$.
\end{lemma}

Consequently, one may say that each almost bipartite graph is either a
K\"{o}nig-Egerv\'{a}ry graph or a $1$-K\"{o}nig-Egerv\'{a}ry graph.

For $X\subseteq V(G)$, the number $\left\vert X\right\vert -\left\vert
N(X)\right\vert $ is the \textit{difference} of $X$, denoted $d(X)$. The
\textit{critical difference} $d(G)$ is $\max\left\{  d(I):I\subseteq
V(G)\right\}  =\max\{d(I):I\in\mathrm{Ind}(G)\}$ \cite{Zhang1990}. If
$A\in\mathrm{Ind}(G)$ with $d\left(  X\right)  =d(G)$, then $A$ is a
\textit{critical independent set} \cite{Zhang1990}. For a graph $G$, let
$\mathrm{MaxCritIndep}(G)=\{S:S$ \textit{is a maximum critical independent
set}$\}$ \cite{LevMan2022c}, $\mathrm{\ker}(G)$ be the intersection of all its
critical independent sets and $\varepsilon(G)=|\mathrm{\ker}(G)|$
\cite{LevMan2012a,LevMan2012c}.

\begin{theorem}
\label{th3}\emph{(i)} \cite{ButTruk2007} Each critical independent set is
included in some $S\in\Omega(G)$.

\emph{(ii)} \cite{Larson2007} Every critical independent set is contained in
some $S\in\mathrm{MaxCritIndep}(G)$.

\emph{(iii)} \cite{Larson2007} There is a matching from $N(S)$ into $S$ for
every critical independent set $S$.
\end{theorem}

\begin{theorem}
\label{th4}\cite{LevMan2012a} For a graph $G$, the following assertions are true:

\emph{(i)} $\mathrm{\ker}(G)\subseteq\emph{core}(G)$;

\emph{(ii)} if $A$ and $B$ are critical in $G$, then $A\cup B$ and $A\cap B$
are critical as well;

\emph{(iii)} $G$ has a unique minimal independent critical set, namely,
$\mathrm{\ker}(G)$.
\end{theorem}

\begin{theorem}
\label{th22}\cite{LevMan2013c,LevMan2022b} If $G$ is bipartite, or an almost
bipartite non-K\"{o}nig-Egerv\'{a}ry graph, then $\mathrm{\ker}%
(G)=\mathrm{core}(G)$.
\end{theorem}

If $S$ is an independent set of a graph $G$ and $A=V\left(  G\right)  -S$,
then we write $G=S\ast A$. For instance, if $E(G[A])=\emptyset$, then $G=S\ast
A$ is bipartite; if $G[A]$ is a complete graph, then $G=S\ast A$ is a split graph.

\begin{theorem}
\label{th715}For a graph $G$, the following properties are equivalent:

\emph{(i)} $G$ is a \textit{K\"{o}nig-Egerv\'{a}ry graph};

\emph{(ii) }\cite{LevMan2002a} $G=S\ast A$, where $S\in$ $\mathrm{Ind}(G)$,
$\left\vert S\right\vert \geq\left\vert A\right\vert $, and $\left(
S,A\right)  $ contains a matching $M$ with $\left\vert M\right\vert
=\left\vert A\right\vert $;

\emph{(iii) }\cite{LevMan2013b} each maximum matching of $G$ matches $V\left(
G\right)  -S$ into $S$, for every $S\in\Omega(G)$;

\emph{(iv)} \cite{LevMan2012a} every $S\in%
\Omega
(G)$ is critical;

\emph{(v)} \cite{Larson2011} there is some $S\in\Omega(G)$, such that $S$ is critical.
\end{theorem}

Some useful properties of K\"{o}nig-Egerv\'{a}ry graphs are recalled in the following.

\begin{theorem}
\label{th2} If $G$ is a \textit{K\"{o}nig-Egerv\'{a}ry graph, then}

\emph{(i)} \cite{LevMan2003} $N\left(  \mathrm{core}(G)\right)  $ is matched
into $\mathrm{core}(G)$ by every maximum matching of $G$;

\emph{(ii) }\cite{LevMan2006} $G-N\left[  \mathrm{core}(G)\right]  $ is a
\textit{K\"{o}nig-Egerv\'{a}ry graph with a perfect matching;}

\emph{(iii)} \cite{LevMan2012a}\textit{ }$\left\vert \mathrm{core}%
(G)\right\vert -\left\vert N\left(  \mathrm{core}(G)\right)  \right\vert
=\alpha\left(  G\right)  -\mu\left(  G\right)  =d\left(  G\right)  $.
\end{theorem}

\begin{corollary}
\label{lem2}If $G$ a \textit{K\"{o}nig-Egerv\'{a}ry graph} without isolated
vertices an\textit{d }$\mathrm{core}(G)=\{v\}$, then $G$ has a perfect
matching and $v$ is a leaf.
\end{corollary}

\begin{proof}
Since $G$ has no isolated vertices, by Theorem \ref{th2}\emph{(i)}, it follows
that $\left\vert N\left(  \mathrm{core}(G)\right)  \right\vert =1$, and hence,
$v$ is a leaf. Further, Theorem \ref{th2}\emph{(ii)} implies $\alpha\left(
G\right)  =\mu\left(  G\right)  $, which ensures that $G$ has a perfect matching.
\end{proof}

\begin{proposition}
\label{prop6}If $G$ is a non-bipartite K\"{o}nig-Egerv\'{a}ry graph, then
there exists an induced subgraph $H$ of $G$, such that $H$ is not a
K\"{o}nig-Egerv\'{a}ry graph.
\end{proposition}

\begin{proof}
Since $G$ is non-bipartite, there exists an odd cycle $C_{2k+1}$ as an induced
subgraph of $G$. Since $\alpha(G[V\left(  C_{2k+1}]\right)  )+\mu(G[V\left(
C_{2k+1}]\right)  )<\left\vert V\left(  C_{2k+1}\right)  \right\vert $, we
infer that $H=C_{2k+1}$ is not a K\"{o}nig-Egerv\'{a}ry graph.
\end{proof}

In other words, Proposition \ref{prop6} says that, unlike bipartite graphs or
forests, being a K\"{o}nig-Egerv\'{a}ry graph is not a hereditary property.

\begin{corollary}
A graph is hereditary K\"{o}nig-Egerv\'{a}ry if and only if it is bipartite.
\end{corollary}

\begin{proof}
Assume, to the contrary, that $G$ is hereditary K\"{o}nig-Egerv\'{a}ry, but is
not a bipartite graph. According to Proposition \ref{prop6}, there is some
subgraph $H$ of $G$, such that $H$ is not a K\"{o}nig-Egerv\'{a}ry graph,
which contradicts the hypothesis on $G$.

The converse is clear, since every subgraph of a bipartite graph is bipartite
as well, and hence it is a K\"{o}nig-Egerv\'{a}ry graph.
\end{proof}

Let $\varrho_{v}\left(  G\right)  $ denote the number of vertices $v\in
V\left(  G\right)  $, such that $G-v$ is a K\"{o}nig-Egerv\'{a}ry graph. For
example, $\varrho_{v}\left(  C_{2k+1}\right)  =2k+1$ for each $k\geq1$, while
$\varrho_{v}\left(  K_{n}\right)  =n-2$, for $n\geq3$.

Let $\varrho_{e}\left(  G\right)  $ equal the number of edges $e\in E\left(
G\right)  $ satisfying $G-e$ is a K\"{o}nig-Egerv\'{a}ry graph. For instance,
$\varrho_{e}\left(  G\right)  =m\left(  G\right)  $, where is $G$ a bipartite
graph or $G=C_{2k+1}$ for each $k\geq1$. Notice that, unlike the bipartite
graphs, deleting an edge of a K\"{o}nig-Egerv\'{a}ry graph may result in a
non-K\"{o}nig-Egerv\'{a}ry graph; for instance, the graph $G_{2}$ from Figure
\ref{fig333} is K\"{o}nig-Egerv\'{a}ry graph, while $G_{2}-ab$ is not
K\"{o}nig-Egerv\'{a}ry graph.

In this paper, we substantially update the current knowledge on critical
independent sets in order to show that for a K\"{o}nig-Egerv\'{a}ry graph $G$:

\begin{itemize}
\item the level of vertex heredity\textbf{ }$\varrho_{v}\left(  G\right)  $ is
equal to\textbf{ }$n\left(  G\right)  -\xi\left(  G\right)  +\varepsilon
\left(  G\right)  $;

\item the level of edge heredity\textbf{ }$\varrho_{e}\left(  G\right)  $ is
greater or equal to $m\left(  G\right)  -\xi\left(  G\right)  +\varepsilon
\left(  G\right)  $ and this bound is tight.
\end{itemize}

\section{Critical independent sets}

\begin{theorem}
\label{Theorem1}Let $A$ be a critical independent set of the graph $G$ and
$X=A\cup N\left(  A\right)  $. Then the following assertions are true:

\emph{(i)} $G\left[  X\right]  $ is a K\"{o}nig-Egerv\'{a}ry graph;

\emph{(ii)} $\alpha\left(  G-X\right)  \leq\mu\left(  G-X\right)  $;

\emph{(iii)} $\alpha(G)=\alpha(G[X])+\alpha(G-X)$;

\emph{(iv)} $\mu\left(  G\left[  X\right]  \right)  +\mu\left(  G-X\right)
=\mu\left(  G\right)  $ and, in particular, each maximum matching of $G\left[
X\right]  $ can be enlarged to a maximum matching of $G$.
\end{theorem}

\begin{proof}
\emph{(i) }By Theorem \ref{th3}(\emph{iii}) and Theorem \ref{th715}%
\textit{(\emph{ii})}, we get that $G\left[  X\right]  $ is a
K\"{o}nig-Egerv\'{a}ry graph having $\alpha(G\left[  X\right]  )=\left\vert
A\right\vert $ and $\mu(G\left[  X\right]  )=\left\vert N(A)\right\vert $.

\emph{(ii) }According to Theorem \ref{th3}\emph{(i)}, there exists a maximum
independent set $S$ such that $A\subseteq S$. Suppose that $\left\vert
B\right\vert >\left\vert N\left(  B\right)  \right\vert $ holds for some
$B\subseteq S-A$. Then, it follows that
\[
\left\vert A\right\vert -\left\vert N\left(  A\right)  \right\vert <\left(
\left\vert A\right\vert -\left\vert N\left(  A\right)  \right\vert \right)
+\left(  \left\vert B\right\vert -\left\vert N\left(  B\right)  \right\vert
\right)  \leq\left\vert A\cup B\right\vert -\left\vert N\left(  A\cup
B\right)  \right\vert ,
\]
which contradicts the hypothesis on $A$, namely, the fact that $\left\vert
A\right\vert -\left\vert N\left(  A\right)  \right\vert =d(G)$. Hence
$\left\vert B\right\vert \leq\left\vert N\left(  B\right)  \right\vert $ is
true for every $B\subseteq S-A$. Consequently, by Hall's Theorem there exists
a matching from $S-A$ into $V\left(  G\right)  -S-N\left(  A\right)  $ that
implies $\left\vert S-A\right\vert \leq\mu\left(  G-X\right)  $.

It remains to show that $\alpha\left(  G-X\right)  =\left\vert S-A\right\vert
$. By way of contradiction, assume that
\[
\alpha\left(  G-X\right)  =\left\vert D\right\vert >\left\vert S-A\right\vert
\]
for some independent set $D\subseteq V-X$. Since $D\cap N\left[  A\right]
=\emptyset$, the set $A\cup D$ is independent, and
\[
\left\vert A\cup D\right\vert =\left\vert A\right\vert +\left\vert
D\right\vert >\left\vert A\right\vert +\left\vert S-A\right\vert
=\alpha\left(  G\right)  ,
\]
which is impossible.

\emph{(iii) }By Part\emph{(i)}, we know\ that $\alpha(G\left[  X\right]
)=\left\vert A\right\vert $. By Theorem \ref{th3}\emph{(i)}, there exists
$S\in\Omega\left(  G\right)  $, such that $A\subseteq S$. Clearly, $S-A$ is an
independent set in $G-X$. Hence, $\alpha(G-X])\geq\left\vert S-A\right\vert $.
Assume that there is an independent set $B$ in $G-X]$, such that $\left\vert
B\right\vert >\left\vert S-A\right\vert $. First, $A\cup B$ is independent in
$G$ and $A\cap B=\emptyset$, since $B\cap X=\emptyset$. Second,
\[
\left\vert A\cup B\right\vert =\left\vert A\right\vert +\left\vert
B\right\vert >\left\vert A\right\vert +\left\vert S-A\right\vert =\left\vert
S\right\vert =\alpha(G),
\]
which is a contradiction. Thus,
\[
\alpha(G\left[  X\right]  )+\alpha(G-X])=\left\vert A\right\vert +\left\vert
S-A\right\vert =\left\vert S\right\vert =\alpha(G).
\]

\emph{(iv) }Let $M_{1}$ be a maximum matching of $H$ and $M_{2}$ be a maximum
matching of $G-X$. We claim that $M_{1}\cup M_{2}$ is a maximum matching of
$G$. \begin{figure}[h]
\setlength{\unitlength}{1.0cm} \begin{picture}(5,4)\thicklines
\put(6.7,1){\oval(11.7,1.5)}
\put(5,1){\oval(4.5,1)}
\put(5,3){\oval(4.5,1)}
\multiput(3.5,1)(1,0){2}{\circle*{0.29}}
\multiput(3.5,3)(1,0){2}{\circle*{0.29}}
\multiput(6.5,1)(0,2){2}{\circle*{0.29}}
\multiput(4.6,1)(0.35,0){5}{\circle*{0.1}}
\multiput(4.6,3)(0.35,0){5}{\circle*{0.1}}
\put(6.5,1){\line(0,1){2}}
\multiput(3.5,1)(1,0){2}{\line(0,1){2}}
\put(0.5,1){\makebox(0,0){$S$}}
\put(0.2,2.2){\makebox(0,0){$G$}}
\put(1.75,1){\makebox(0,0){$S-A$}}
\put(1.5,3){\makebox(0,0){$V-S-N(A)$}}
\put(10,1){\oval(3.3,1)}
\put(9.5,3){\oval(2.3,1)}
\multiput(9,1)(1,0){3}{\circle*{0.29}}
\multiput(9,3)(1,0){2}{\circle*{0.29}}
\put(9,1){\line(0,1){2}}
\put(10,1){\line(0,1){2}}
\put(12,1){\makebox(0,0){$A$}}
\put(11.2,3){\makebox(0,0){$N(A)$}}
\put(8.2,2.2){\makebox(0,0){$M_{1}$}}
\put(2.5,2.2){\makebox(0,0){$M_{2}$}}
\end{picture}\caption{$S\in\Omega(G)$ and $A$ is a critical independent set of
$G${.}}%
\label{fig44}%
\end{figure}
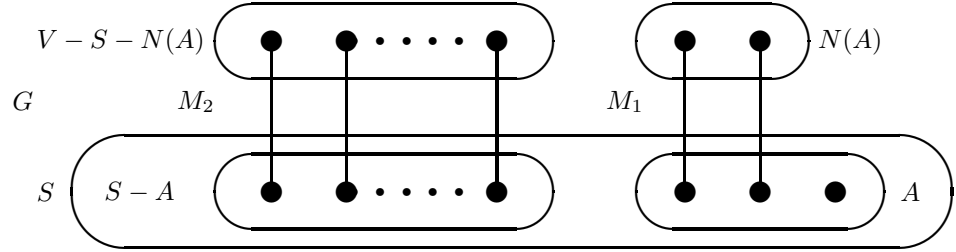

The only edges that may enlarge $M_{1}\cup M_{2}$ belong to the set $\left(
N\left(  A\right)  ,V\left(  G\right)  -S-N\left(  A\right)  \right)  $. The
matching $M_{1}$ covers all the vertices of $N\left(  A\right)  $ in
accordance with Theorem \ref{th715}\emph{(iii)} and Part \emph{(i)}.
Therefore, to choose an edge from the set $\left(  N\left(  A\right)
,V\left(  G\right)  -S-N\left(  A\right)  \right)  $ means to loose an edge
from $M_{1}$. In other words, no matching different from $M_{1}\cup M_{2}$ may
overstep $\left\vert M_{1}\cup M_{2}\right\vert $.

Consequently, each maximum matching of $G\left[  X\right]  $ can find its
counterpart in $G-X$ in order to build a maximum matching of $G$.
\end{proof}

Let us mention that Part \emph{(iii) }and Part \emph{(i) }of Theorem
\ref{Theorem1} generalize Part \emph{(i) }and Part \emph{(iii) }of Theorem
\ref{th10} for an arbitrary critical independent set.

\begin{theorem}
\label{th10}\cite{Larson2011} For any graph $G$, there is a unique set
$X\subseteq V(G)$ such that

\emph{(i)} $\alpha(G)=\alpha\left(  G\left[  X\right]  \right)  +\alpha\left(
G-X\right)  $;

\emph{(ii)} $X=N\left[  A\right]  $\ for every $A\in\mathrm{MaxCritIndep}(G)$; \ 

\emph{(iii)} $G\left[  X\right]  $ is a \textit{K\"{o}nig-Egerv\'{a}ry} graph.

\emph{(iv) }$G-X$ has only $\emptyset$ as a critical independent set.
\end{theorem}

\begin{remark}
The graph $H$ has the only $\emptyset$ as a critical independent set if and
only if $\left\vert N(A)\right\vert >\left\vert A\right\vert $ for every
independent set $A$, i.e., $H$ is regularizable \cite{Berge1982}.
\end{remark}

For instance, $K_{2p}$ has $\emptyset$ as a critical independent set, while
its every vertex is $\mu$-critical.

Theorem \ref{Theorem1} allows us to give an alternative proof of the following
inequality due to Lorentzen.

\begin{corollary}
\label{cor1}\cite{Lorentzen1966,Schrijver2003,LevMan2012c}, The inequality
$d\left(  G\right)  \geq\alpha\left(  G\right)  -\mu\left(  G\right)  $ holds
for every graph $G$.
\end{corollary}

\begin{proof}
Let $A$ be a critical independent set of $G$, and $X=A\cup N\left(  A\right)
$.

By Theorem \ref{Theorem1}\emph{(ii)}, we get $\alpha\left(  G-X\right)
-\mu\left(  G-X\right)  \leq0$. Hence it follows that%
\[
\alpha\left(  G\left[  X\right]  \right)  -\mu\left(  G\left[  X\right]
\right)  \geq\left(  \alpha\left(  G\left[  X\right]  \right)  +\alpha\left(
G-X\right)  \right)  -\left(  \mu\left(  G\left[  X\right]  \right)
+\mu\left(  G-X\right)  \right)  .
\]

Theorem \ref{Theorem1}\emph{(iii) }claims that $\mu\left(  G\left[  X\right]
\right)  +\mu\left(  G-X\right)  =\mu\left(  G\right)  $.

Since $A$ is a critical independent set, there exists some $S\in\Omega\left(
G\right)  $ such that $A\subseteq S$, and $\alpha\left(  G\left[  X\right]
\right)  =\left\vert A\right\vert $, by Theorem \ref{th3}\emph{(i)}. Hence we
have%
\[
\alpha\left(  G\left[  X\right]  \right)  +\alpha\left(  G-X\right)
=\left\vert A\right\vert +\left\vert S-A\right\vert =\alpha\left(  G\right)
.
\]
In addition,\emph{ }Theorem \ref{Theorem1}\emph{(i)} and Theorem
\ref{th715}\emph{(iii) }imply that $\mu\left(  G\left[  X\right]  \right)
=\left\vert N(A)\right\vert $.

Finally, we obtain
\begin{align*}
d\left(  G\right)   &  =\max\left\{  \left\vert I\right\vert -\left\vert
N(I)\right\vert :I\in\mathrm{Ind}(G)\right\}  =\left\vert A\right\vert
-\left\vert N(A)\right\vert =\\
&  =\alpha\left(  G\left[  X\right]  \right)  -\mu\left(  G\left[  X\right]
\right)  \geq\alpha\left(  G\right)  -\mu\left(  G\right)  ,
\end{align*}
and this completes the proof.
\end{proof}

The graph $G$ from Figure \ref{fig24} has $\mathrm{\ker}(G)=\left\{
a,b,c\right\}  $. Notice that $\mathrm{\ker}(G)\subseteq\emph{core}(G)$;
$S=\left\{  a,b,c,v\right\}  $ is a largest critical independent set, and
neither $S\subseteq\mathrm{core}(G)$ nor $\emph{core}(G)\subseteq S$. In
addition, $\mathrm{core}(G)$ is not a critical set of $G$. 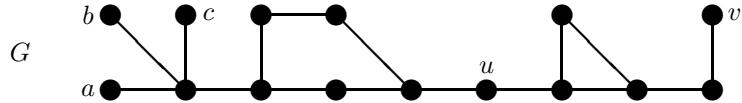
\begin{figure}[h]
\setlength{\unitlength}{1cm}\begin{picture}(5,1.2)\thicklines
\multiput(3,0)(1,0){9}{\circle*{0.29}}
\multiput(3,1)(1,0){4}{\circle*{0.29}}
\multiput(9,1)(2,0){2}{\circle*{0.29}}
\put(3,0){\line(1,0){8}}
\put(4,0){\line(0,1){1}}
\put(3,1){\line(1,-1){1}}
\put(5,0){\line(0,1){1}}
\put(5,1){\line(1,0){1}}
\put(6,1){\line(1,-1){1}}
\put(9,0){\line(0,1){1}}
\put(9,1){\line(1,-1){1}}
\put(11,0){\line(0,1){1}}
\put(2.7,0){\makebox(0,0){$a$}}
\put(2.7,1){\makebox(0,0){$b$}}
\put(4.3,1){\makebox(0,0){$c$}}
\put(8,0.3){\makebox(0,0){$u$}}
\put(11.3,1){\makebox(0,0){$v$}}
\put(1.8,0.5){\makebox(0,0){$G$}}
\end{picture}\caption{$G$ is a non-K\"{o}nig-Egerv\'{a}ry graph with
\textrm{core}$(G)=\{a,b,c,u\}$.}%
\label{fig24}%
\end{figure}

\begin{proposition}
\cite{LevMan2013a}\label{prop9} For a vertex $v$ in a graph $G$, the following
assertions hold:

\emph{(i)} $d(G-v)=d(G)-1$ if and only if $v\in\ker(G)$;

\emph{(ii)} if $v\in\ker(G)$, then $\ker(G-v)\subseteq\ker(G)-\{v\}$.
\end{proposition}

\begin{theorem}
\cite{LevMan2013a} \label{Th11} For a critical independent set $A$ in a graph
$G$, the equality $A=$ $\ker(G)$ holds if and only if for each $v\in A$ there
exists a matching from $N(A)$ into $A-\left\{  v\right\}  $.
\end{theorem}

\begin{lemma}
\label{lem4}If $A$ is a vertex subset of a graph $G$, and $C\subseteq
B\subseteq A$ are such that $N\left(  B\right)  =N\left(  B-C\right)  $, then
$N\left(  A\right)  =N\left(  A-C\right)  $.
\end{lemma}

\begin{proof}
The result follows from the equalities
\[
N\left(  A-C\right)  =N\left(  B-C\right)  \cup N\left(  A-B\right)  =N\left(
B\right)  \cup N\left(  A-B\right)  =N\left(  A\right)  .
\]

\end{proof}

\begin{corollary}
\label{cor118}If $v\in\ker(G)\subseteq A$, then $N\left(  A\right)  =N\left(
A-\left\{  v\right\}  \right)  $.
\end{corollary}

\begin{proof}
Since $v\in\ker(G)$, by Theorem \ref{Th11}, $N\left(  \ker(G)\right)
=N\left(  \ker(G)-\left\{  v\right\}  \right)  $. Now, by Lemma \ref{lem4}, we
obtain $N\left(  A-\left\{  v\right\}  \right)  =N\left(  A\right)  $.
\end{proof}

\begin{proposition}
\label{prop13}Let $H=G\left[  N\left[  A\right]  \right]  $, where
$\ker\left(  G\right)  \subseteq A\subseteq V\left(  G\right)  $ for some
graph $G$. Then the following assertions are true:

\emph{(i) }$d\left(  H\right)  \geq d\left(  G\right)  $.

\emph{(ii) }If $A$ is a critical independent set in $G$, then $d\left(
H\right)  =d\left(  G\right)  $ and $\ker\left(  H\right)  =\ker\left(
G\right)  $.
\end{proposition}

\begin{proof}
\emph{(i) }By the definition\emph{ }of the graph $H$, we know that
\[
N_{H}\left[  \ker\left(  G\right)  \right]  =N_{G}\left[  \ker\left(
G\right)  \right]  \subseteq N_{G}\left[  A\right]  =V\left(  H\right)  .
\]

Hence, we obtain
\[
d\left(  H\right)  =d_{H}\left(  \ker\left(  H\right)  \right)  \geq
d_{H}\left(  \ker\left(  G\right)  \right)  =d_{G}\left(  \ker\left(
G\right)  \right)  =d\left(  G\right)  .
\]

\emph{(ii) }By Theorem \ref{Theorem1}(\emph{i}), $H$ is a
K\"{o}nig-Egerv\'{a}ry graph. Therefore,%
\[
d\left(  G\right)  =d_{G}\left(  A\right)  =d_{H}\left(  A\right)  =d\left(
H\right)  .
\]
Thus, $d\left(  G\right)  =d\left(  H\right)  $. Consequently, $\ker\left(
G\right)  $ is a critical set in $H$ (because $d_{H}\left(  \ker\left(
G\right)  \right)  =d\left(  H\right)  $), which implies $\ker\left(
H\right)  \subseteq\ker\left(  G\right)  $. Moreover, $\ker\left(  H\right)  $
is a critical set in $G$ (since $d_{G}\left(  \ker\left(  H\right)  \right)
=d\left(  G\right)  $), and this implies $\ker\left(  G\right)  \subseteq
\ker\left(  H\right)  $. Therefore, $\ker\left(  H\right)  =\ker\left(
G\right)  $, as required.
\end{proof}

If $G$ is a K\"{o}nig-Egerv\'{a}ry graph, then, by Theorem \ref{th4}%
\emph{(ii)} and\emph{ }Theorem\emph{ }\ref{th715}\emph{(iii)}, $\emph{core}%
\left(  G\right)  $ is critical in $G$. Together with Proposition
\ref{prop13}\emph{(ii) }it implies the following.

\begin{corollary}
Let $H=G\left[  N\left[  \mathrm{core}\left(  G\right)  \right]  \right]  $.
Then the equalities $d\left(  H\right)  =d\left(  G\right)  $ and $\ker\left(
H\right)  =\ker\left(  G\right)  $ hold for every K\"{o}nig-Egerv\'{a}ry graph
$G$.
\end{corollary}

Consider the graphs in Figure \ref{Fig1444}: $\left\{  x,y,z\right\}  $ is
independent in $G_{1}$ and $d\left(  G_{1}\left[  N\left[  \left\{
x,y,z\right\}  \right]  \right]  \right)  >d\left(  G_{1}\right)  $, while
$\left\{  a,b,c\right\}  $ is independent in $G_{2}$\ and $d\left(
G_{2}\left[  N\left[  \left\{  a,b,c\right\}  \right]  \right]  \right)
>d\left(  G_{2}\right)  $.

\begin{figure}[h]
\setlength{\unitlength}{1cm}\begin{picture}(5,1.2)\thicklines
\multiput(3,0)(1,0){5}{\circle*{0.29}}
\multiput(4,1)(1,0){3}{\circle*{0.29}}
\put(3,0){\line(1,0){4}}
\put(4,0){\line(0,1){1}}
\put(5,0){\line(0,1){1}}
\put(5,1){\line(1,0){1}}
\put(6,1){\line(1,-1){1}}
\put(3,0.35){\makebox(0,0){$x$}}
\put(3.65,1){\makebox(0,0){$y$}}
\put(4.65,0.35){\makebox(0,0){$z$}}
\put(2.3,0.5){\makebox(0,0){$G_{1}$}}
\multiput(9,0)(1,0){4}{\circle*{0.29}}
\multiput(10,1)(1,0){3}{\circle*{0.29}}
\put(9,0){\line(1,0){3}}
\put(10,0){\line(0,1){1}}
\put(11,1){\line(1,0){1}}
\put(11,0){\line(0,1){1}}
\put(12,0){\line(0,1){1}}
\put(9,0.35){\makebox(0,0){$a$}}
\put(9.65,1){\makebox(0,0){$b$}}
\put(10.65,0.35){\makebox(0,0){$c$}}
\put(8.3,0.5){\makebox(0,0){$G_{2}$}}
\end{picture}\caption{$d\left(  G_{1}\right)  =d\left(  G_{2}\right)  =1$,
while only $G_{1}$ is not a K\"{o}nig-Egerv\'{a}ry graph.}%
\label{Fig1444}%
\end{figure}
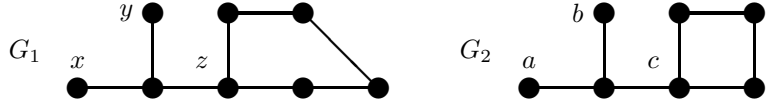

\begin{lemma}
\label{lem6}Let $A$ and $S$ be independent sets in a graph $G$, such that $A$
is critical and $A\subseteq S$. Then there is a matching from $S-A$ into
$V\left(  G\right)  -S-N(A)$.
\end{lemma}

\begin{proof}
Let $B\subseteq S-A$. Then $\left\vert A\cup B\right\vert -\left\vert N(A\cup
B)\right\vert \leq\left\vert A\right\vert -\left\vert N(A)\right\vert $,
because $A$ is critical. Hence,
\[
\left\vert B\right\vert =\left\vert A\cup B\right\vert -\left\vert
A\right\vert \leq\left\vert N(A\cup B)\right\vert -\left\vert N(A)\right\vert
\leq\left\vert N(B)\cap\left(  V\left(  G\right)  -S-N(A)\right)  \right\vert
,
\]
which is the Hall condition between $S-A$ and $V\left(  G\right)  -S-N(A)$.
Consequently, there is a matching from $S-A$ into $V\left(  G\right)  -S-N(A)$.
\end{proof}

\begin{proposition}
\label{prop17}Let $v\in\mathrm{core}\left(  G\right)  -\ker\left(  G\right)
$. Then the following assertions are true:

\emph{(i)} $\ker\left(  G\right)  $ is critical in $G-v$;

\emph{(ii)} $d\left(  G-v\right)  =d\left(  G\right)  $;

\emph{(iii)} $d\left(  G-v\right)  =\alpha\left(  G\right)  -\mu\left(
G\right)  $, if in addition, $G$ is a K\"{o}nig-Egerv\'{a}ry graph.
\end{proposition}

\begin{proof}
\emph{(i) }Suppose, on the contrary, that $d_{G-v}(B)>d_{G-v}\left(
\ker\left(  G\right)  \right)  $ for some independent set $B\subseteq V\left(
G\right)  -\left\{  v\right\}  $. First, $v\notin N\left[  \ker\left(
G\right)  \right]  $, because $v\in\mathrm{core}\left(  G\right)  -\ker\left(
G\right)  $. Hence, $d_{G}\left(  \ker\left(  G\right)  \right)
=d_{G-v}\left(  \ker\left(  G\right)  \right)  $. Clearly, $d_{G}(B)+1\geq
d_{G-v}(B)$ for every independent set $B\subseteq V\left(  G\right)  -\left\{
v\right\}  $. Therefore, $B$ should be critical in $G$, otherwise,
\[
d_{G}(B)<d(G)=d_{G}\left(  \ker\left(  G\right)  \right)  =d_{G-v}\left(
\ker\left(  G\right)  \right)  <d_{G-v}(B)\leq d_{G}(B)+1,
\]
in other words,%
\[
d_{G}(B)+2\leq d(G)+1=d_{G}\left(  \ker\left(  G\right)  \right)
+1=d_{G-v}\left(  \ker\left(  G\right)  \right)  +1<d_{G-v}(B)+1\leq
d_{G}(B)+2,
\]
which is impossible.

Thus, by Theorem \ref{th3}, there exists $S\in\Omega\left(  G\right)  $, such
that $B\subseteq S$. It implies that $v\notin N_{G}\left[  B\right]  $, i.e.,
\[
d_{G-v}\left(  B\right)  =d_{G}\left(  B\right)  =d_{G}\left(  \ker\left(
G\right)  \right)  =d_{G-v}\left(  \ker\left(  G\right)  \right)  ,
\]
which contradicts the hypothesis on $B$.

\emph{(ii)} Numerically, Part \emph{(i) }says that\emph{ }$d\left(
G-v\right)  =d\left(  G\right)  $.

\emph{(iii)} Since $G$ is a K\"{o}nig-Egerv\'{a}ry graph, by Theorem
\ref{th2}\emph{(iii) }we know that $d\left(  G\right)  =\alpha\left(
G\right)  -\mu\left(  G\right)  $. Now Part \emph{(ii)} completes the proof.
\end{proof}

\section{The vertex heredity equality}

\begin{proposition}
Let $G$ be a K\"{o}nig-Egerv\'{a}ry graph and $v\in V(G)$ be such that $G-v$
is still a K\"{o}nig-Egerv\'{a}ry graph. Then $v\in\emph{core}\left(
G\right)  $ if and only if there exists a maximum matching that does not
saturate $v$.
\end{proposition}

\begin{proof}
Since $v\in$ $\emph{core}\left(  G\right)  $, it follows that $\alpha
(G-v)=\alpha(G)-1$. Consequently, we have $\alpha(G)+%
\mu
(G)-1=n(G)-1=n(G-v)=\alpha(G-v)+%
\mu
(G-v)$ which implies that $%
\mu
(G)=%
\mu
(G-v)$. In other words, there is a maximum matching in $G$ not saturating $v$.

Conversely, suppose that there exists a maximum matching in $G$ that does not
saturate $v$. Since, by Theorem \ref{th2}\emph{(i)}, $N(\emph{core}\left(
G\right)  )$ is matched into $\emph{core}\left(  G\right)  $ by every maximum
matching, it follows that $v\notin N(\emph{core}\left(  G\right)  )$. Assume
that $v\notin\emph{core}\left(  G\right)  $. By Theorem \ref{th2}\emph{(ii)},
$H=G-N[\emph{core}\left(  G\right)  ]$ is a K\"{o}nig-Egerv\'{a}ry graph, $H$
has a perfect matching and every maximum matching $M$ of $G$ is of the form
$M=M_{1}\cup M_{2}$, where $M_{1}$ matches $N(\emph{core}\left(  G\right)  )$
into $\emph{core}\left(  G\right)  $, while $M_{2}$ is a perfect matching of
$H$. Consequently, $v$ is saturated by every maximum matching of $G$, in
contradiction with the hypothesis on $v$.
\end{proof}

Consider the graphs from Figure \ref{fig55}. Notice that $G_{1}-v_{2}%
,G_{1}-v_{3},G_{1}-v_{4}$ are K\"{o}nig-Egerv\'{a}ry graphs, but $G_{1}-v_{1}$
is not a K\"{o}nig-Egerv\'{a}ry graph, while $G_{3}-x$ is a
K\"{o}nig-Egerv\'{a}ry graph for each $x\in V\left(  G_{3}\right)  $.

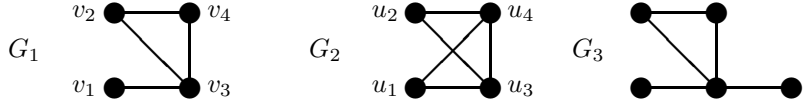
\begin{figure}[h]
\setlength{\unitlength}{1cm}\begin{picture}(5,1.2)\thicklines
\multiput(3,0)(1,0){2}{\circle*{0.29}}
\multiput(3,1)(1,0){2}{\circle*{0.29}}
\put(3,0){\line(1,0){1}}
\put(3,1){\line(1,0){1}}
\put(3,1){\line(1,-1){1}}
\put(4,0){\line(0,1){1}}
\put(2.6,0){\makebox(0,0){$v_{1}$}}
\put(2.6,1){\makebox(0,0){$v_{2}$}}
\put(4.4,0){\makebox(0,0){$v_{3}$}}
\put(4.4,1){\makebox(0,0){$v_{4}$}}
\put(1.8,0.5){\makebox(0,0){$G_{1}$}}
\multiput(7,0)(1,0){2}{\circle*{0.29}}
\multiput(7,1)(1,0){2}{\circle*{0.29}}
\put(7,0){\line(1,0){1}}
\put(7,0){\line(1,1){1}}
\put(7,1){\line(1,0){1}}
\put(7,1){\line(1,-1){1}}
\put(8,0){\line(0,1){1}}
\put(6.6,0){\makebox(0,0){$u_{1}$}}
\put(6.6,1){\makebox(0,0){$u_{2}$}}
\put(8.4,0){\makebox(0,0){$u_{3}$}}
\put(8.4,1){\makebox(0,0){$u_{4}$}}
\put(5.8,0.5){\makebox(0,0){$G_{2}$}}
\multiput(10,0)(1,0){3}{\circle*{0.29}}
\multiput(10,1)(1,0){2}{\circle*{0.29}}
\put(10,0){\line(1,0){2}}
\put(10,1){\line(1,0){1}}
\put(10,1){\line(1,-1){1}}
\put(11,0){\line(0,1){1}}
\put(9.3,0.5){\makebox(0,0){$G_{3}$}}
\end{picture}\caption{$G_{1}$, $G_{2}$ and $G_{3}$ are K\"{o}nig-Egerv\'{a}ry
graphs}%
\label{fig55}%
\end{figure}

\begin{theorem}
\label{th9}If $G$ is a K\"{o}nig-Egerv\'{a}ry graph, then the following
assertions are true:

\emph{(i)} $G-v$ is a K\"{o}nig-Egerv\'{a}ry graph for every $v\in V\left(
G\right)  -\mathrm{core}\left(  G\right)  $;

\emph{(ii)} $G-v$ is not a K\"{o}nig-Egerv\'{a}ry graph for every
$v\in\mathrm{core}\left(  G\right)  -\ker\left(  G\right)  $, moreover, $G-v$
is $1$-K\"{o}nig-Egerv\'{a}ry;

\emph{(iii)} $G-v$ is a K\"{o}nig-Egerv\'{a}ry graph for every $v\in
\ker\left(  G\right)  $;

\emph{(iv)} $\varrho_{v}\left(  G\right)  =n\left(  G\right)  -\xi\left(
G\right)  +\varepsilon\left(  G\right)  $.
\end{theorem}

\begin{proof}
\textit{\emph{(i)}} Since $v\in V\left(  G\right)  -$ $\emph{core}\left(
G\right)  $, there exists some $S\in\Omega\left(  G\right)  $ such that $v\in
V\left(  G\right)  -S$.

By Theorem \ref{th715}\emph{(iii)},\ each maximum matching of $G$ matches
$V\left(  G\right)  -S$ into $S$, and consequently, $\mu(G)=\left\vert
V\left(  G\right)  -S\right\vert $.

Let $H=G-v$. Then $S\in\Omega\left(  H\right)  $ and $\mu\left(  H\right)
=\left\vert V\left(  G\right)  -S-\left\{  v\right\}  \right\vert $. Thus
\[
\alpha\left(  H\right)  +\mu\left(  H\right)  =\alpha\left(  G\right)
+\mu\left(  G\right)  -1=n\left(  G\right)  -1=\left\vert V\left(  H\right)
\right\vert .
\]
Hence, $H=G-v$ is a K\"{o}nig-Egerv\'{a}ry graph for each $v\in V\left(
G\right)  -$ $\mathrm{core}\left(  G\right)  $.

\emph{(ii)} Suppose that $G-v$ is K\"{o}nig-Egerv\'{a}ry. By Theorem
\ref{th2}\emph{(iii)} and Proposition \ref{prop17}\emph{(iii)},
\[
\alpha\left(  G-v\right)  -\mu\left(  G-v\right)  =d\left(  G-v\right)
=\alpha\left(  G\right)  -\mu\left(  G\right)  .
\]

Clearly, $\alpha\left(  G-v\right)  =\alpha\left(  G\right)  -1$, because
$v\in\mathrm{core}\left(  G\right)  $.

Hence,%
\[
\alpha\left(  G\right)  -1-\mu\left(  G-v\right)  =\alpha\left(  G\right)
-\mu\left(  G\right)  .
\]

Thus, $\mu\left(  G-v\right)  =\mu\left(  G\right)  -1$. Consequently,%
\[
n\left(  G-v\right)  =\alpha\left(  G-v\right)  +\mu\left(  G-v\right)
=\alpha\left(  G\right)  -1+\mu\left(  G\right)  -1=n\left(  G\right)
-2=n\left(  G-v\right)  -1\text{,}%
\]
which is impossible. Therefore, $G-v$ is not a K\"{o}nig-Egerv\'{a}ry graph.

If $\mu\left(  G-v\right)  =\mu\left(  G\right)  $, then
\[
\alpha\left(  G-v\right)  +\mu\left(  G-v\right)  =\alpha\left(  G\right)
-1+\mu\left(  G\right)  =n\left(  G\right)  -1=n\left(  G-v\right)  ,
\]
which is a contradiction, because $G-v$ is not K\"{o}nig-Egerv\'{a}ry. Thus,
$\mu\left(  G-v\right)  =\mu\left(  G\right)  -1$. Consequently, we have
\[
\alpha\left(  G-v\right)  +\mu\left(  G-v\right)  =\alpha\left(  G\right)
-1+\mu\left(  G\right)  -1=n\left(  G\right)  -2=n\left(  G-v\right)
-1\text{,}%
\]
which means that $G-v$ is a $1$-K\"{o}nig-Egerv\'{a}ry graph.

\emph{(iii)} By Proposition \ref{prop9}\emph{(i)}, $d(G-v)=d(G)-1$, because
$v\in\ker(G)$. Since $v\in\mathrm{core}\left(  G\right)  $, we get
$\alpha\left(  G-v\right)  =\alpha\left(  G\right)  -1$.

Let $S\in\Omega\left(  G\right)  $. Then $v\in S$ and $S-v\in\Omega\left(
G-v\right)  $, because $v\in\mathrm{core}\left(  G\right)  $.

By definition, $d_{G}(S)=\left\vert S\right\vert -\left\vert N_{G}\left(
S\right)  \right\vert $. Moreover, $d(G)=$ $d_{G}(S)$, because, by Theorem
\ref{th715}\emph{(iv)}, each maximum independent set of a
K\"{o}nig-Egerv\'{a}ry graph is critical.

It is clear that $N_{G-v}\left(  S-v\right)  =N_{G}\left(  S-v\right)  $.
Taking in account Corollary \ref{cor118}, we obtain%
\begin{gather*}
d_{G-v}(S-v)=\left\vert S-v\right\vert -\left\vert N_{G-v}\left(  S-v\right)
\right\vert =\left\vert S\right\vert -1-\left\vert N_{G}\left(  S-v\right)
\right\vert =\\
\left\vert S\right\vert -\left\vert N_{G}\left(  S\right)  \right\vert
-1+\left\vert N_{G}\left(  S\right)  \right\vert -\left\vert N_{G}\left(
S-v\right)  \right\vert =\\
d_{G}(S)-1+\left\vert N_{G}\left(  S\right)  \right\vert -\left\vert
N_{G}\left(  S-v\right)  \right\vert =d(G-v).
\end{gather*}

In conclusion, $S-v$ is critical in $G-v$. By Theorem \ref{th715}\emph{(v)},
$G-v$ is a K\"{o}nig-Egerv\'{a}ry graph.

\emph{(iv)} It directly follows from Parts \emph{(i),(ii),(iii)}.
\end{proof}

Theorem \ref{th9} directly implies that $\varrho_{v}\left(  G\right)  >0$ for
every K\"{o}nig-Egerv\'{a}ry graph $G$ of order at least $1$. In other words,
we arrive at the following.

\begin{corollary}
If $G$ is a K\"{o}nig-Egerv\'{a}ry graph with $n(G)\geq1$, then there exists a
vertex $v$ such that $G-v$ is K\"{o}nig-Egerv\'{a}ry as well.
\end{corollary}

Clearly, for bipartite graphs $\varrho_{v}\left(  G\right)  =n(G)$. Notice
that the equality from Theorem \ref{th9}\emph{(iv)} holds also for some
non-K\"{o}nig-Egerv\'{a}ry graphs; for instance, $\varrho_{v}\left(
C_{2k+1}\right)  =2k+1=n\left(  C_{2k+1}\right)  -\xi\left(  C_{2k+1}\right)
+\varepsilon\left(  C_{2k+1}\right)  $. However, there exist
non-K\"{o}nig-Egerv\'{a}ry graphs for which the equality from Theorem
\ref{th9}\emph{(iv)} is not true. For instance, consider the graphs from
Figure \ref{fig9}: $\mathrm{core}(G_{1})=\ker\left(  G_{1}\right)  =\emptyset$
and $\varrho_{v}\left(  G_{1}\right)  =4$; $\mathrm{core}(G_{2})=\ker\left(
G_{2}\right)  =\left\{  u_{1},u_{2}\right\}  $ and $\varrho_{v}\left(
G_{2}\right)  =3$, while $\varrho\left(  G_{2}\right)  =0$.

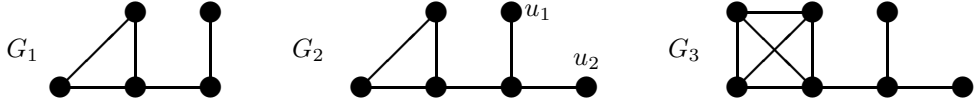
\begin{figure}[h]
\setlength{\unitlength}{1cm}\begin{picture}(5,1.2)\thicklines
\multiput(1,0)(1,0){3}{\circle*{0.29}}
\multiput(2,1)(1,0){2}{\circle*{0.29}}
\put(1,0){\line(1,0){2}}
\put(1,0){\line(1,1){1}}
\put(2,0){\line(0,1){1}}
\put(3,0){\line(0,1){1}}
\put(0.5,0.5){\makebox(0,0){$G_{1}$}}
\multiput(5,0)(1,0){4}{\circle*{0.29}}
\multiput(6,1)(1,0){2}{\circle*{0.29}}
\put(5,0){\line(1,0){3}}
\put(5,0){\line(1,1){1}}
\put(6,0){\line(0,1){1}}
\put(7,0){\line(0,1){1}}
\put(7.35,1){\makebox(0,0){$u_{1}$}}
\put(8,0.35){\makebox(0,0){$u_{2}$}}
\put(4.3,0.5){\makebox(0,0){$G_{2}$}}
\multiput(10,0)(1,0){4}{\circle*{0.29}}
\multiput(10,1)(1,0){3}{\circle*{0.29}}
\put(10,0){\line(1,0){3}}
\put(10,0){\line(1,1){1}}
\put(10,1){\line(1,0){1}}
\put(10,1){\line(1,-1){1}}
\put(11,0){\line(0,1){1}}
\put(12,0){\line(0,1){1}}
\put(10,0){\line(0,1){1}}
\put(9.3,0.5){\makebox(0,0){$G_{3}$}}
\end{picture}\caption{$G_{1}$, $G_{2}$ and $G_{3}$ are
non-K\"{o}nig-Egerv\'{a}ry graphs}%
\label{fig9}%
\end{figure}

\begin{proposition}
\cite{LevMan2013b}\label{prop10} Let $G$ be a K\"{o}nig-Egerv\'{a}ry graph and
$v\in V(G)$ be such that $G-v$ is still a K\"{o}nig-Egerv\'{a}ry graph. Then
$v\in\mathrm{core}(G)$ if and only if $v$ is not $\mu$-critical.
\end{proposition}

Theorem \ref{th9} and Proposition \ref{prop10} imply the following.

\begin{corollary}
\label{cor99}Let $G$ be a K\"{o}nig-Egerv\'{a}ry graph and $v\in V(G)$. Then
$v\in\ker(G)$ if and only if $v\in\mathrm{core}(G)$ and $v$ is not $\mu$-critical.
\end{corollary}

\begin{proposition}
\label{prop3}\cite{LevMan2006} If $G$ is a K\"{o}nig-Egerv\'{a}ry graph, then
$\eta\left(  G\right)  \leq\alpha\left(  G\right)  -\xi\left(  G\right)  $.
\end{proposition}

\begin{proposition}
\label{prop4}\cite{BGL2002} If $G$ is a graph without isolated vertices, then
$\alpha\left(  G\right)  -\mu\left(  G\right)  \leq\xi\left(  G\right)  $.
\end{proposition}

\begin{corollary}
\label{cor11}If $G$ is a K\"{o}nig-Egerv\'{a}ry graph, then%
\[
\eta\left(  G\right)  +\mu\left(  G\right)  +\varepsilon\left(  G\right)
\leq\varrho_{v}\left(  G\right)  \leq2\mu\left(  G\right)  +\varepsilon\left(
G\right)  ,
\]
and the above inequalities turn out to be equalities if and only if
$\eta\left(  G\right)  =\mu\left(  G\right)  $.
\end{corollary}

\begin{proof}
By Proposition \ref{prop3}, we get
\[
\eta\left(  G\right)  +\mu\left(  G\right)  +\varepsilon\left(  G\right)
\leq\alpha\left(  G\right)  +\mu\left(  G\right)  -\xi\left(  G\right)
+\varepsilon\left(  G\right)  =n\left(  G\right)  -\xi\left(  G\right)
+\varepsilon\left(  G\right)  =\varrho_{v}\left(  G\right)  .
\]
Further, using Proposition \ref{prop4} and the fact that $G$ is a
K\"{o}nig-Egerv\'{a}ry graph, we obtain%
\begin{gather*}
\varrho_{v}\left(  G\right)  =n\left(  G\right)  -\xi\left(  G\right)
+\varepsilon\left(  G\right)  =\alpha\left(  G\right)  +\mu\left(  G\right)
-\xi\left(  G\right)  +\varepsilon\left(  G\right)  \leq\\
\alpha\left(  G\right)  +\mu\left(  G\right)  -\left(  \alpha\left(  G\right)
-\mu\left(  G\right)  \right)  +\varepsilon\left(  G\right)  =2\mu\left(
G\right)  +\varepsilon\left(  G\right)  ,
\end{gather*}
and this completes the proof.
\end{proof}

Notice that: $\eta\left(  K_{4}-e\right)  +\mu\left(  K_{4}-e\right)
+\varepsilon\left(  K_{4}-e\right)  =\varrho_{v}\left(  K_{4}-e\right)
$\ \ and $\varrho_{v}\left(  C_{4}\right)  =2\mu\left(  C_{4}\right)
+\varepsilon\left(  C_{4}\right)  $. In other words, the bounds in Corollary
\ref{cor11} are tight.

\section{The edge heredity inequality}

\begin{lemma}
\label{lem5}Suppose that $B\subseteq A$ are critical independent sets in a
graph $G$. Then there is a perfect matching between $A-B$ and $N(A)-N(B)$,
and, consequently, $G\left[  N\left[  A-B\right]  \right]  $ is a
K\"{o}nig-Egerv\'{a}ry graph.
\end{lemma}

\begin{proof}
Since $A$ and $B$ are critical in $G$, the graphs $H_{1}=G\left[  N\left[
A\right]  \right]  $ and $H_{2}=G\left[  N\left[  B\right]  \right]  $ are
K\"{o}nig-Egerv\'{a}ry, in accordance with Theorem \ref{Theorem1}\emph{(i)}.
Thus, we get
\[
\left\vert A\right\vert -\left\vert N(A)\right\vert =\left\vert B\right\vert
-\left\vert N(B)\right\vert \Leftrightarrow\left\vert A\right\vert -\left\vert
B\right\vert =\left\vert N(A)\right\vert -\left\vert N(B)\right\vert .
\]

By Theorem \ref{th3}\emph{(iii)}, every maximum matching of $H_{1}$ matches
$N\left(  A\right)  $ into $A$, and every maximum matching of $H_{2}$ matches
$N\left(  B\right)  $ into $B$. On the other hand, there is no edge $xy\in
E\left(  G\right)  $, such that $x\in B$ and $y\in N\left(  A\right)
-N\left(  B\right)  $. \ Therefore, every maximum matching $M_{1}$ of $H_{1}$
matches $N\left(  A\right)  -N\left(  B\right)  $ into $A-B$. Since
$\left\vert A\right\vert -\left\vert B\right\vert =\left\vert N(A)\right\vert
-\left\vert N(B)\right\vert $, we infer that the restriction $M_{2}$ of
$M_{1}$ to $N\left(  A\right)  -N\left(  B\right)  $\ defines a perfect
matching between $N\left(  A\right)  -N\left(  B\right)  $ and $A-B$.

Therefore,
\[
\mu\left(  G\left[  N\left[  A-B\right]  \right]  \right)  =\left\vert
M_{2}\right\vert =\left\vert A-B\right\vert =\alpha\left(  G\left[  N\left[
A-B\right]  \right]  \right)  ,
\]
which means, by Theorem \ref{th715}, that $G\left[  N\left[  A-B\right]
\right]  $ is a K\"{o}nig-Egerv\'{a}ry graph, because $A-B$ is independent.
\end{proof}

\begin{figure}[h]
\setlength{\unitlength}{1cm}\begin{picture}(5,1.2)\thicklines
\multiput(1,0)(1,0){4}{\circle*{0.29}}
\multiput(2,1)(1,0){2}{\circle*{0.29}}
\put(1,0){\line(1,0){3}}
\put(2,0){\line(0,1){1}}
\put(2,1){\line(1,-1){1}}
\put(2,1){\line(1,0){1}}
\put(0.5,0.5){\makebox(0,0){$G_{1}$}}
\multiput(6,0)(1,0){3}{\circle*{0.29}}
\multiput(6,1)(1,0){3}{\circle*{0.29}}
\put(6,0){\line(1,0){2}}
\put(6,0){\line(0,1){1}}
\put(7,1){\line(1,0){1}}
\put(7,1){\line(1,-1){1}}
\put(8,0){\line(0,1){1}}
\put(6.35,1){\makebox(0,0){$u_{1}$}}
\put(7,0.35){\makebox(0,0){$u_{2}$}}
\put(5.3,0.5){\makebox(0,0){$G_{2}$}}
\multiput(10,0)(1,0){4}{\circle*{0.29}}
\multiput(11,1)(1,0){3}{\circle*{0.29}}
\put(10,0){\line(1,0){3}}
\put(11,0){\line(0,1){1}}
\put(12,0){\line(1,1){1}}
\put(12,0){\line(0,1){1}}
\put(13,0){\line(0,1){1}}
\put(10,0.3){\makebox(0,0){$v_{1}$}}
\put(10.65,1){\makebox(0,0){$v_{2}$}}
\put(11.65,1){\makebox(0,0){$v_{3}$}}
\put(13.36,1){\makebox(0,0){$v_{4}$}}
\put(9.3,0.5){\makebox(0,0){$G_{3}$}}
\end{picture}\caption{$G_{1}$, $G_{2}$ and $G_{3}$ are K\"{o}nig-Egerv\'{a}ry
graphs}%
\label{fig33}%
\end{figure}
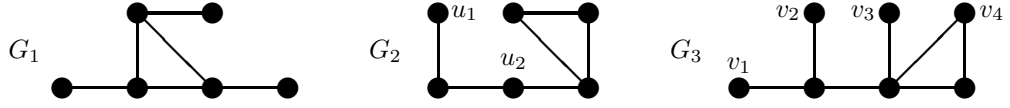

Consider the graphs from Figure \ref{fig33}: $G_{1}$ and $G_{2}$ have perfect
matchings,
\[
\ker(G_{1})=\mathrm{core}\left(  G_{1}\right)  =\emptyset\text{ and }%
\ker\left(  G_{2}\right)  =\emptyset\subset\mathrm{core}\left(  G_{2}\right)
=\left\{  u_{1},u_{2}\right\}  ,
\]
while $\ker\left(  G_{3}\right)  =\left\{  v_{1},v_{2}\right\}  \subset
\mathrm{core}\left(  G_{2}\right)  =\left\{  v_{1},v_{2},v_{3}\right\}  $\ and
$G_{3}-N\left[  \ker(G_{3})\right]  $ has a perfect matching.

\begin{theorem}
\label{cor9}Assume that $G$ is a K\"{o}nig-Egerv\'{a}ry graph and $S\in
\Omega\left(  G\right)  $. Then every maximum matching $M$ of $G$ defines a
perfect matching between:

\emph{(i) }$N\left(  S\right)  -N\left(  \mathrm{core}\left(  G\right)
\right)  =V\left(  G\right)  -S-N\left(  \mathrm{core}\left(  G\right)
\right)  $ and $S-\mathrm{core}\left(  G\right)  $;

\emph{(ii)} $N\left(  \mathrm{core}\left(  G\right)  \right)  -N\left(
\ker(G)\right)  $ and $\mathrm{core}\left(  G\right)  -\ker(G)$;

\emph{(iii)} $N\left(  S\right)  -N\left(  \ker(G)\right)  =V\left(  G\right)
-S-N\left(  \ker(G)\right)  $ and $S-\ker(G)$.
\end{theorem}

\begin{proof}
It is known that every maximum independent set in a K\"{o}nig-Egerv\'{a}ry
graph is critical (by Theorem \ref{th715}). Moreover, $\mathrm{core}\left(
G\right)  $ is critical, as being the intersection of a family of critical
sets (according to Theorem \ref{th4}\emph{(ii)}). Now, Lemma \ref{lem5}
applied to the following pairs of critical sets: $\mathrm{core}\left(
G\right)  \subseteq S$, $\ker(G)\subseteq\mathrm{core}\left(  G\right)  $, and
$\ker(G)\subseteq S$ completes the proof.
\end{proof}

Combining\ Theorem \ref{th2}\emph{(i),(ii)} and Theorem \ref{cor9}%
\emph{(ii),(iii)} we infer the following.

\begin{corollary}
\label{cor110}If $G$ is a K\"{o}nig-Egerv\'{a}ry graph, then every maximum
matching of the subgraph $G\left[  N\left[  \ker\left(  G\right)  \right]
\right]  $ is included in a maximum matching of $G$. Moreover, the trace of
every maximum matching in $G$ is a maximum matching in $N\left[  \ker\left(
G\right)  \right]  $.
\end{corollary}

\begin{proposition}
\label{prop18} For every graph $G$, its number of $\mu$-critical edges is less
or equal to $\mu\left(  G\right)  $.
\end{proposition}

\begin{proof}
Let $M$ be a maximum matching of $G$. Hence, no edge from $E\left(  G\right)
-M$ is $\mu$-critical. In other words, the only edges from $M$ may be $\mu$-critical.
\end{proof}

\begin{theorem}
\cite{LevMan2006}\label{cor999} Let $G$ be a K\"{o}nig-Egerv\'{a}ry graph. If
$e\notin(\mathrm{core}(G),N(\mathrm{core}(G))$, then $G-e$ is a
K\"{o}nig-Egerv\'{a}ry graph, as well. In particular, if $\mathrm{core}%
(G)=\emptyset$, then $G-e$ is a K\"{o}nig-Egerv\'{a}ry graph for every edge
$e$.
\end{theorem}

\begin{proposition}
\label{prop15}There are no $\alpha$-critical edges of a graph $G$ in
$(N(\mathrm{core}(G),N\left(  N(\mathrm{core}(G)\right)  )$.
\end{proposition}

\begin{proof}
If $e=xy,x\in N(\mathrm{core}(G),y\in N\left(  N(\mathrm{core}(G)\right)  $ is
$\alpha$-critical in $G$, then there exists a maximum independent set $S$ in
$G$ such that $x\in S$. Since $x\in N(\mathrm{core}(G)$, there is some vertex
$z\in\mathrm{core}(G)$ and $zx\in E\left(  G\right)  $. Hence, $z\notin S$,
which is impossible, because $z\in\mathrm{core}(G)$.
\end{proof}

\begin{proposition}
\cite{LevMan2013a}\label{prop14} For every edge $e\in\left(  \ker\left(
G\right)  ,N\left(  \ker\left(  G\right)  \right)  \right)  $, there exist two
matchings $M_{1}$ and $M_{2}$ from $N\left(  \ker\left(  G\right)  \right)  $
into $\ker\left(  G\right)  $ such that $e\notin M_{1}$ and $e\in M_{2}$.
\end{proposition}

\begin{theorem}
\label{prop1}If $G$ is a K\"{o}nig-Egerv\'{a}ry graph, then

\emph{(i)} there are no $\mu$-critical edges of $G$ in $(\ker\left(  G\right)
,N(\ker\left(  G\right)  )$;

\emph{(ii)} there are no $\mu$-critical edges of $G$ in $(\mathrm{core}%
(G)-\ker\left(  G\right)  ,N(\ker\left(  G\right)  )$;

\emph{(iii)} the only possible location for $\mu$-critical edges of $G$ in
$(\mathrm{core}(G),N(\mathrm{core}(G))$ is
\[
\left(  \mathrm{core}(G)-\ker\left(  G\right)  ,N\left(  \mathrm{core}%
(G)\right)  -N\left(  \ker\left(  G\right)  \right)  \right)  ;
\]

\emph{(iv)} the maximum possible number of $\mu$-critical edges of $G$ in
$\left(  \mathrm{core}(G),N(\mathrm{core}(G)\right)  $ is equal to $\xi\left(
G\right)  -\varepsilon\left(  G\right)  $;

\emph{(v)} $\varrho_{e}\left(  G\right)  \geq m\left(  G\right)  -\xi\left(
G\right)  +\varepsilon\left(  G\right)  $.
\end{theorem}

\begin{proof}
\emph{(i) }By Proposition \ref{prop14}, for every edge $e\in\left(
\ker\left(  G\right)  ,N(\ker\left(  G\right)  \right)  $, there is a maximum
matching in $N\left[  \ker\left(  G\right)  \right]  $ not including $e$. In
other words, the edge $e$ is not $\mu$-critical in $N\left[  \ker\left(
G\right)  \right]  $. Consequently, $e$ is not $\mu$-critical in $G$ as well,
by Corollary \ref{cor110}.

\emph{(ii) }Since both $\mathrm{core}(G)$ and $\ker\left(  G\right)  $ are
critical independent sets in $G$ and $\ker\left(  G\right)  \subseteq
\mathrm{core}(G)$, there exist a matching $M_{1}$ from $N\left(  \ker\left(
G\right)  \right)  $ into $\ker\left(  G\right)  $, a matching $M_{2}$ from
$N\left(  \mathrm{core}(G)\right)  $ into $\mathrm{core}(G)$, and a maximum
matching $M_{3}$ of $G$, such that $M_{1}\subseteq M_{2}\subseteq M_{3}$, in
accordance with Theorem\ \ref{Theorem1}\emph{(iv)}, Theorem \ref{th2}%
\emph{(i)} and Corollary \ref{cor110}. Hence, if $e\in(\mathrm{core}%
(G)-\ker\left(  G\right)  ,$ $N(\ker\left(  G\right)  )$, then $e\notin M_{3}%
$, and, consequently, the edge $e$ is not $\mu$-critical.

\emph{(iii)} It follows from Parts \emph{(i) }and\emph{ (ii).}

\emph{(iv) }By Theorem \ref{cor9}\emph{(ii)}, there is a perfect matching
contained in
\[
\left(  \mathrm{core}(G)-\ker\left(  G\right)  ,N\left(  \mathrm{core}%
(G)\right)  -N\left(  \ker\left(  G\right)  \right)  \right)  .
\]
By Proposition \ref{prop18} and Theorem \ref{cor9}\emph{(ii)}, we complete the proof.

\emph{(v) }Clearly, if the edge $e$ is not $\alpha$-critical, then $G-e$ is
K\"{o}nig-Egerv\'{a}ry if and only if $e$ is not $\mu$-critical. Thus, the
lower bound on $\varrho_{e}\left(  G\right)  $ follows by combining Theorem
\ref{cor999}, Proposition \ref{prop15} and Part \emph{(iv)}.
\end{proof}

\begin{corollary}
\label{cor88} If $G$ is a K\"{o}nig-Egerv\'{a}ry graph and $\mathrm{core}%
(G)=\ker\left(  G\right)  $, then $\varrho_{e}\left(  G\right)  =m\left(
G\right)  $.
\end{corollary}

\begin{proof}
If $\mathrm{core}(G)=\ker\left(  G\right)  $, then $\xi\left(  G\right)
=\varepsilon\left(  G\right)  $, and, consequently, by Theorem \ref{prop1}%
\emph{(v)},
\[
m\left(  G\right)  \geq\varrho_{e}\left(  G\right)  \geq m\left(  G\right)
-\xi\left(  G\right)  +\varepsilon\left(  G\right)  =m\left(  G\right)  +0,
\]
which means that $\varrho_{e}\left(  G\right)  =m\left(  G\right)  $.
\end{proof}

The graph $K_{4}-e$ shows that the converse of Corollary \ref{cor88} is not true.

There are K\"{o}nig-Egerv\'{a}ry graphs with $\xi\left(  G\right)  \neq0$,
while $\varrho_{e}\left(  G\right)  =m\left(  G\right)  $; for example, the
graph $G_{1}$ from Figure \ref{fig333}. Nevertheless, $\mathrm{core}%
(G_{1})=\ker(G_{1})$.

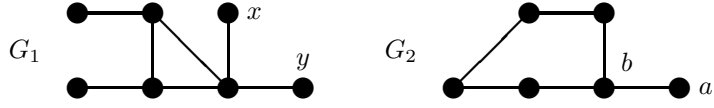
\begin{figure}[h]
\setlength{\unitlength}{1cm}\begin{picture}(5,1.2)\thicklines
\multiput(3,0)(1,0){4}{\circle*{0.29}}
\multiput(3,1)(1,0){3}{\circle*{0.29}}
\put(3,0){\line(1,0){3}}
\put(3,1){\line(1,0){1}}
\put(4,0){\line(0,1){1}}
\put(5,0){\line(0,1){1}}
\put(4,1){\line(1,-1){1}}
\put(5.35,1){\makebox(0,0){$x$}}
\put(6,0.35){\makebox(0,0){$y$}}
\put(2.3,0.5){\makebox(0,0){$G_{1}$}}
\multiput(8,0)(1,0){4}{\circle*{0.29}}
\multiput(9,1)(1,0){2}{\circle*{0.29}}
\put(8,0){\line(1,1){1}}
\put(8,0){\line(1,0){3}}
\put(9,1){\line(1,0){1}}
\put(10,0){\line(0,1){1}}
\put(11.35,0){\makebox(0,0){$a$}}
\put(10.3,0.35){\makebox(0,0){$b$}}
\put(7.3,0.5){\makebox(0,0){$G_{2}$}}
\end{picture}\caption{K\"{o}nig-Egerv\'{a}ry graphs with $\mathrm{core}%
(G_{1})=\left\{  x,y\right\}  $ and $\mathrm{core}(G_{2})=\left\{  a\right\}
$.}%
\label{fig333}%
\end{figure}

\begin{corollary}
\label{cor2}If $G$ is a K\"{o}nig-Egerv\'{a}ry graph, then
\[
\varrho_{e}\left(  G\right)  =m\left(  G\right)  -\xi\left(  G\right)
+\varepsilon\left(  G\right)  ,
\]
if and only
\[
\left(  \mathrm{core}(G)-\ker\left(  G\right)  ,N\left(  \mathrm{core}%
(G)\right)  -N\left(  \ker\left(  G\right)  \right)  \right)
\]
contains a unique perfect matching.
\end{corollary}

\begin{proof}
If $M$ is the unique perfect matching contained in
\[
\left(  \mathrm{core}(G)-\ker\left(  G\right)  ,N\left(  \mathrm{core}%
(G)\right)  -N\left(  \ker\left(  G\right)  \right)  \right)  ,
\]
then every its edge is $\mu$-critical in $G$.

If there are more than one perfect matching in
\[
\left(  \mathrm{core}(G)-\ker\left(  G\right)  ,N\left(  \mathrm{core}%
(G)\right)  -N\left(  \ker\left(  G\right)  \right)  \right)  ,
\]
say $M_{1},M_{2}$ an so on, then $M_{1}-$\ $M_{2}\neq\emptyset$, and,
consequently, the number of $\mu$-critical edges is less or equal to
$\xi\left(  G\right)  -\varepsilon\left(  G\right)  -1$.
\end{proof}

If $G$ is a K\"{o}nig-Egerv\'{a}ry graph with $\xi\left(  G\right)
-\varepsilon\left(  G\right)  =1$, then a perfect matching between
$\mathrm{core}(G)-\ker\left(  G\right)  $ and $N\left(  \mathrm{core}%
(G)\right)  -N\left(  \ker\left(  G\right)  \right)  $ is, clearly, unique.
Hence, Corollary \ref{cor2} directly implies the following.

\begin{corollary}
\label{cor114}If $G$ is a K\"{o}nig-Egerv\'{a}ry graph and $\varepsilon\left(
G\right)  =\xi\left(  G\right)  -1$, then
\[
\varrho_{e}\left(  G\right)  =m\left(  G\right)  -1.
\]

\end{corollary}

In accordance with Theorem \ref{th9}\emph{(iv)}, Theorem \ref{prop1}\emph{(v)}
and Corollary \ref{cor114} we obtain the following.

\begin{theorem}
If $G$ is a K\"{o}nig-Egerv\'{a}ry graph, then $n\left(  G\right)
-\varrho_{v}\left(  G\right)  \geq m\left(  G\right)  -\varrho_{e}\left(
G\right)  $, while $n\left(  G\right)  -\varrho_{v}\left(  G\right)  =m\left(
G\right)  -\varrho_{e}\left(  G\right)  $ for $\xi\left(  G\right)
-\varepsilon\left(  G\right)  \leq1$.
\end{theorem}

If for every two incident edges of a cycle $C$ exactly one of them belongs to
a matching $M$, then $C$ is called an $M$-alternating cycle \cite{Krohdal1977}%
. It is clear that an $M$-alternating cycle should be of even length. A
matching $M$ in $G$ is called alternating cycle-free if $G$ has no
$M$-alternating cycle.

\begin{lemma}
\cite{GolumbicHL2001}\label{lem8} Let $M$ be a perfect matching in $G$. Then
$M$ is a unique perfect matching if and only if $M$ is alternating cycle-free
in $G$.
\end{lemma}

\begin{theorem}
\label{cor11212}The number of $\mu$-critical edges of a K\"{o}nig-Egerv\'{a}ry
graph $G$ belonging to $(\mathrm{core}(G),N(\mathrm{core}(G))$, i.e.,
$m\left(  G\right)  -\varrho_{e}\left(  G\right)  $, may be either $\xi\left(
G\right)  -\varepsilon\left(  G\right)  $ or any integer in $\left\{
0,\cdots,\xi\left(  G\right)  -\varepsilon\left(  G\right)  -2\right\}  $,
whenever $\xi\left(  G\right)  -\varepsilon\left(  G\right)  \geq2$.
\end{theorem}

\begin{proof}
By Corollary \ref{cor2} and Proposition \ref{prop1}\emph{(iii)}, the number of
$\mu$-critical edges of $G$ equals $\xi\left(  G\right)  -\varepsilon\left(
G\right)  $ if and only if there is a unique perfect matching in
\[
\left(  \mathrm{core}(G)-\ker\left(  G\right)  ,N\left(  \mathrm{core}%
(G)\right)  -N\left(  \ker\left(  G\right)  \right)  \right)  .
\]

For instance, consider a family of graphs $H_{k}=C_{2k+1}+e,k\geq1$, where $e$
is an edge connecting a vertex, say $v$, from $C_{2k+1}$ with a leaf, say $w$.
Clearly, $\mathrm{core}(H_{k})=\left\{  w\right\}  $, $N\left(  \mathrm{core}%
(H_{k})\right)  =\left\{  v\right\}  $, $\ker\left(  H_{k}\right)  =N\left(
\ker\left(  H_{k}\right)  \right)  =\emptyset$, $\left\{  vw\right\}  $ is a
unique perfect matching between $\mathrm{core}(H_{k})-\ker\left(
H_{k}\right)  $ and $N\left(  \mathrm{core}(H_{k})\right)  -N\left(
\ker\left(  H_{k}\right)  \right)  $, $\xi\left(  H_{k}\right)  -\varepsilon
\left(  H_{k}\right)  =1$, and $\left\{  vw\right\}  $ is the only $\mu
$-critical edge of $H_{k}$ belonging to $(\mathrm{core}(H_{k}),N(\mathrm{core}%
(H_{k}))$.

If $\left(  \mathrm{core}(G)-\ker\left(  G\right)  ,N\left(  \mathrm{core}%
(G)\right)  -N\left(  \ker\left(  G\right)  \right)  \right)  $ has more than
one\ perfect matching, then by Lemma \ref{lem8}, there are alternating cycles
starting with the size $4$ at least. It means that there are no graphs with
the number of $\mu$-critical edges equal to $\xi\left(  G\right)
-\varepsilon\left(  G\right)  -1$.

To complete the proof, let us show that for every integer $\lambda\in\left\{
0,\cdots,\xi\left(  G\right)  -\varepsilon\left(  G\right)  -2\right\}  $,
where $\xi\left(  G\right)  -\varepsilon\left(  G\right)  \geq2$, there exists
a K\"{o}nig-Egerv\'{a}ry graph $G$, such that the number of its $\mu$-critical
edges belonging to $(\mathrm{core}(G),N(\mathrm{core}(G))$ equals $\lambda
$.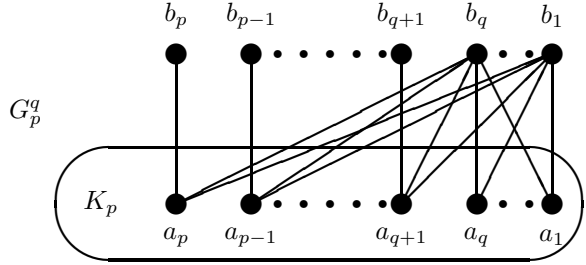
\begin{figure}[h]
\setlength{\unitlength}{1.0cm} \begin{picture}(5,4)\thicklines
\multiput(5,1)(1,0){2}{\circle*{0.29}}
\multiput(5,3)(1,0){2}{\circle*{0.29}}
\multiput(8,1)(1,0){3}{\circle*{0.29}}
\multiput(8,3)(1,0){3}{\circle*{0.29}}
\multiput(6,1)(0.35,0){6}{\circle*{0.1}}
\multiput(6,3)(0.35,0){6}{\circle*{0.1}}
\put(5,1){\line(0,1){2}}
\put(6,1){\line(0,1){2}}
\put(8,1){\line(0,1){2}}
\put(9,1){\line(0,1){2}}
\put(10,1){\line(0,1){2}}
\put(5,1){\line(2,1){4}}
\put(6,1){\line(2,1){4}}
\put(5,1){\line(5,2){5}}
\put(6,1){\line(3,2){3}}
\put(8,1){\line(1,2){1}}
\put(8,1){\line(1,1){2}}
\put(9,1){\line(1,2){1}}
\put(9,3){\line(1,-2){1}}
\put(6.9,1){\oval(7,1.5)}
\put(10,0.55){\makebox(0,0){$a_{1}$}}
\put(10,3.5){\makebox(0,0){$b_{1}$}}
\multiput(9,1)(0.35,0){3}{\circle*{0.1}}
\multiput(9,3)(0.35,0){3}{\circle*{0.1}}
\put(9,0.55){\makebox(0,0){$a_{q}$}}
\put(9,3.5){\makebox(0,0){$b_{q}$}}
\put(8,0.55){\makebox(0,0){$a_{q+1}$}}
\put(8,3.5){\makebox(0,0){$b_{q+1}$}}
\put(6,0.55){\makebox(0,0){$a_{p-1}$}}
\put(6,3.5){\makebox(0,0){$b_{p-1}$}}
\put(5,0.55){\makebox(0,0){$a_{p}$}}
\put(5,3.5){\makebox(0,0){$b_{p}$}}
\put(4,1){\makebox(0,0){$K_{p}$}}
\put(3,2.2){\makebox(0,0){$G_{p}^{q}$}}
\end{picture}\caption{A family of graphs $G_{p}^{q}$ illustrating the second
conclusion of Corollary \ref{cor11212}{.}}%
\label{fig11}%
\end{figure}

Consider the family of K\"{o}nig-Egerv\'{a}ry graphs $G_{p}^{q},p\geq
q\geq2,p-q=\lambda$, such that $G_{p}^{q}\left[  \left\{  a_{1},a_{2}%
,...,a_{p}\right\}  \right]  =K_{p}$, and $\left\{  b_{1},b_{2},...,b_{p}%
\right\}  $ is an independent set in $G_{p}^{q}$ (see Figure \ref{fig11}).
Clearly, $\mathrm{core}(G_{p}^{q})=\left\{  b_{1},b_{2},...,b_{p}\right\}  $
and $\ker\left(  G_{p}^{q}\right)  =\emptyset$. Since $N\left[  \mathrm{core}%
(G_{p}^{q})\right]  =V\left(  G_{p}^{q}\right)  $ and no edge of $G_{p}%
^{q}\left[  \left\{  a_{1},a_{2},...,a_{p}\right\}  \right]  =K_{p}$ is $\mu
$-critical in $G_{p}^{q}$, we get that the number of $\mu$-critical edges of
$G_{p}^{q}$ located in
\[
\left(  \mathrm{core}(G_{p}^{q})-\ker\left(  G_{p}^{q}\right)  ,N\left(
\mathrm{core}(G_{p}^{q})\right)  -N\left(  \ker\left(  G_{p}^{q}\right)
\right)  \right)  =\left(  \mathrm{core}(G_{p}^{q}),N\left(  \mathrm{core}%
(G_{p}^{q})\right)  \right)
\]
equals $p-q$, which means that $\lambda=p-q=\xi\left(  G_{p}^{q}\right)
-\varepsilon\left(  G_{p}^{q}\right)  -q$.
\end{proof}

\begin{figure}[h]
\setlength{\unitlength}{1cm}\begin{picture}(5,1.2)\thicklines
\multiput(5,0)(1,0){4}{\circle*{0.29}}
\multiput(6,1)(1,0){3}{\circle*{0.29}}
\put(5,0){\line(1,0){3}}
\put(6,0){\line(0,1){1}}
\put(7,0){\line(0,1){1}}
\put(7,0){\line(1,1){1}}
\put(7,1){\line(1,-1){1}}
\put(8,0){\line(0,1){1}}
\put(4.65,0){\makebox(0,0){$a$}}
\put(5.65,1){\makebox(0,0){$b$}}
\put(7.35,1){\makebox(0,0){$c$}}
\put(8.35,1){\makebox(0,0){$d$}}
\put(4,0.5){\makebox(0,0){$G$}}
\end{picture}\caption{$G$ is a K\"{o}nig-Egerv\'{a}ry graph with $\ker\left(
G\right)  =\left\{  a,b\right\}  $ and $\mathrm{core}(G)=\left\{
a,b,c,d\right\}  $}%
\label{fig45}%
\end{figure}
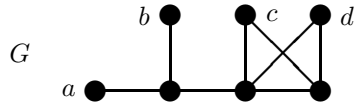

It is worth mentioning that the graphs $G_{p}^{q}$ presented in Theorem
\ref{cor11212} show that $m\left(  G\right)  -\xi\left(  G\right)
+\varepsilon\left(  G\right)  $ is a tight bound for $\varrho_{e}\left(
G\right)  $. On the other hand, there exist examples with $m\left(  G\right)
-\varrho_{e}\left(  G\right)  \in\left\{  \xi\left(  G\right)  -\varepsilon
\left(  G\right)  -2,\cdots,0\right\}  $, when $\varepsilon\left(  G\right)
\neq0$ (see Figure \ref{fig45}, where $m\left(  G\right)  =\varrho_{e}\left(
G\right)  =8>m\left(  G\right)  -\xi\left(  G\right)  +\varepsilon\left(
G\right)  =6$).

In addition, notice that the graph $G_{3}$ from Figure \ref{fig33} satisfies
$m\left(  G_{3}\right)  =\varrho_{e}\left(  G_{3}\right)  +1=7=m\left(
G_{3}\right)  -\xi\left(  G_{3}\right)  +\varepsilon\left(  G_{3}\right)  $,
while $\ker\left(  G_{3}\right)  =\left\{  v_{1},v_{2}\right\}  $ and
$\mathrm{core}(G_{3})=\left\{  v_{1},v_{2},v_{3}\right\}  $.

\begin{corollary}
\label{cor116}If $G$ is a K\"{o}nig-Egerv\'{a}ry graph, then $m\left(
G\right)  -\xi\left(  G\right)  +\varepsilon\left(  G\right)  \geq\eta\left(
G\right)  $.
\end{corollary}

\begin{proof}
First, let us prove that $m\left(  G\right)  +\varepsilon\left(  G\right)
\geq\alpha\left(  G\right)  $.

\textit{Case 1}. $G$ is connected and $m\left(  G\right)  \neq0$.
\[
m\left(  G\right)  +\varepsilon\left(  G\right)  \geq n\left(  G\right)
-1\geq\alpha\left(  G\right)  ,
\]
because if $\alpha\left(  G\right)  =n\left(  G\right)  $, then $m\left(
G\right)  =0$.

\textit{Case 2}. $G$ is disconnected. If a connected components of $G$, say
$H$, has non-empty set of edges, then, by Case 1, $m\left(  H\right)
+\varepsilon\left(  H\right)  \geq\alpha\left(  H\right)  $. Otherwise,
$H=K_{1}$. Hence, $\varepsilon\left(  H\right)  =\alpha\left(  H\right)  =1$.
Thus, $m\left(  H\right)  +\varepsilon\left(  H\right)  \geq\alpha\left(
H\right)  $ as well. Therefore, the graph $G$ itself satisfies the inequality
$m\left(  G\right)  +\varepsilon\left(  G\right)  \geq\alpha\left(  G\right)
$.

Second, by Proposition \ref{prop3}, we obtain%
\[
m\left(  G\right)  -\xi\left(  G\right)  +\varepsilon\left(  G\right)
\geq\alpha\left(  G\right)  -\xi\left(  G\right)  \geq\eta\left(  G\right)  ,
\]
which completes the proof.
\end{proof}

Theorem \ref{prop1}\emph{(v)} and Corollary \ref{cor116} imply the following.

\begin{corollary}
If $G$ is a K\"{o}nig-Egerv\'{a}ry graph, then $\varrho_{e}\left(  G\right)
\geq\eta\left(  G\right)  $.
\end{corollary}

\section{Conclusions}

Our main findings, Theorem \ref{th9}\emph{(iv) }and Theorem \ref{prop1}%
\emph{(v)}, motivate the following.

\begin{problem}
Express $\varrho_{e}\left(  G\right)  $ using various graph invariants of
K\"{o}nig-Egerv\'{a}ry graphs.
\end{problem}

\begin{problem}
Characterize non-K\"{o}nig-Egerv\'{a}ry \textit{graphs satisfying;}

\begin{itemize}
\item $\varrho_{e}\left(  G\right)  \geq m\left(  G\right)  -\xi\left(
G\right)  +\varepsilon\left(  G\right)  $;

\item $\varrho_{v}\left(  G\right)  =n\left(  G\right)  -\xi\left(  G\right)
+\varepsilon\left(  G\right)  $;

\item $\varrho_{v}\left(  G\right)  =\alpha\left(  G\right)  +\mu\left(
G\right)  -\xi\left(  G\right)  +\varepsilon\left(  G\right)  $.
\end{itemize}
\end{problem}

\begin{figure}[h]
\setlength{\unitlength}{1cm}\begin{picture}(5,1.5)\thicklines
\multiput(3,0.5)(1,0){3}{\circle*{0.29}}
\multiput(3,1.5)(1,0){3}{\circle*{0.29}}
\put(3,0.5){\line(0,1){1}}
\put(3,1.5){\line(2,-1){2}}
\put(3,0.5){\line(1,0){2}}
\put(3,0.5){\line(1,1){1}}
\put(4,0.5){\line(1,1){1}}
\put(4,0.5){\line(0,1){1}}
\put(5,0.5){\line(0,1){1}}
\qbezier(3,0.5)(4,-0.3)(5,0.5)
\put(2.65,1.5){\makebox(0,0){$x$}}
\put(4.35,1.5){\makebox(0,0){$y$}}
\put(5.35,1.5){\makebox(0,0){$z$}}
\put(2,1){\makebox(0,0){$G_{1}$}}
\multiput(8,0.5)(1,0){3}{\circle*{0.29}}
\multiput(8,1.5)(1,0){4}{\circle*{0.29}}
\put(10,0.5){\line(1,1){1}}
\put(8,0.5){\line(0,1){1}}
\put(8,1.5){\line(2,-1){2}}
\put(8,0.5){\line(1,0){2}}
\put(8,0.5){\line(1,1){1}}
\put(9,0.5){\line(1,1){1}}
\put(9,0.5){\line(0,1){1}}
\put(10,0.5){\line(0,1){1}}
\qbezier(8,0.5)(9,-0.3)(10,0.5)
\put(7.65,1.5){\makebox(0,0){$t$}}
\put(9.35,1.5){\makebox(0,0){$u$}}
\put(10.35,1.5){\makebox(0,0){$v$}}
\put(11.35,1.5){\makebox(0,0){$w$}}
\put(7,1){\makebox(0,0){$G_{2}$}}
\end{picture}\caption{$\ker(G_{1})=\emptyset\subset\mathrm{core}%
(G_{1})=\left\{  x,y,z\right\}  $, while $\ker(G_{2})=\mathrm{core}%
(G_{2})=\left\{  t,u,v,w\right\}  $.}%
\label{fig88}%
\end{figure}
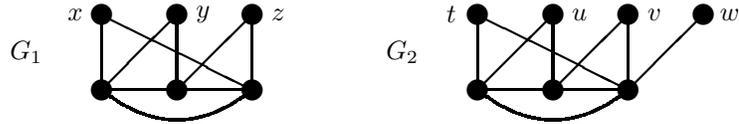

Using Theorem \ref{th9}\emph{(iv)} and Corollary \ref{cor88}, we may conclude
with the following.

\begin{theorem}
\label{Theorem100}If $G$ is a K\"{o}nig-Egerv\'{a}ry graph, then:

\emph{(i)} $\varrho_{v}\left(  G\right)  =n\left(  G\right)  $ if and only if
$\mathrm{core}\left(  G\right)  =\ker(G)$;

(\emph{ii}) if $\mathrm{core}\left(  G\right)  =\ker(G)$, then $\varrho
_{e}\left(  G\right)  =m\left(  G\right)  $.
\end{theorem}

The converse of Theorem \ref{Theorem100}(\emph{ii}) is not true. For example,
consider the K\"{o}nig-Egerv\'{a}ry graph $G_{1}$ from Figure \ref{fig88}:
$\varrho_{e}\left(  G_{1}\right)  =m\left(  G_{1}\right)  $ and $\mathrm{core}%
\left(  G_{1}\right)  \neq\ker(G_{1})$. Nevertheless, there are
K\"{o}nig-Egerv\'{a}ry graphs like $G_{2}$ from Figure \ref{fig88}, which
enjoy both equalities $\varrho_{e}\left(  G_{2}\right)  =m\left(
G_{2}\right)  $ and $\mathrm{core}\left(  G_{2}\right)  =\ker(G_{2})$.

On the other hand, consider the unicyclic K\"{o}nig-Egerv\'{a}ry graphs from
Figure \ref{fig2}: both $G_{1}-v$ and $G_{1}-uv$ are not
K\"{o}nig-Egerv\'{a}ry graphs, while $\varrho_{v}\left(  G_{2}\right)
=n\left(  G_{2}\right)  $ and $\varrho_{e}\left(  G_{2}\right)  =m\left(
G_{2}\right)  $.

It motivates the following.

\begin{problem}
Characterize K\"{o}nig-Egerv\'{a}ry \textit{graphs, where }$\varrho_{e}\left(
G\right)  =m\left(  G\right)  $\textit{ implies }$\mathrm{core}\left(
G\right)  =\ker(G)$.
\end{problem}

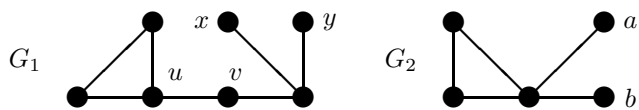
\begin{figure}[h]
\setlength{\unitlength}{1cm}\begin{picture}(5,1.2)\thicklines
\multiput(3,0)(1,0){4}{\circle*{0.29}}
\multiput(4,1)(1,0){3}{\circle*{0.29}}
\put(3,0){\line(1,0){3}}
\put(3,0){\line(1,1){1}}
\put(4,0){\line(0,1){1}}
\put(6,0){\line(0,1){1}}
\put(5,1){\line(1,-1){1}}
\put(4.65,1){\makebox(0,0){$x$}}
\put(5.1,0.3){\makebox(0,0){$v$}}
\put(6.35,1){\makebox(0,0){$y$}}
\put(4.3,0.3){\makebox(0,0){$u$}}
\put(2.3,0.5){\makebox(0,0){$G_{1}$}}
\multiput(8,0)(1,0){3}{\circle*{0.29}}
\multiput(8,1)(2,0){2}{\circle*{0.29}}
\put(8,0){\line(0,1){1}}
\put(8,1){\line(1,-1){1}}
\put(8,0){\line(1,0){2}}
\put(9,0){\line(1,1){1}}
\put(10.35,1){\makebox(0,0){$a$}}
\put(10.35,0){\makebox(0,0){$b$}}
\put(7.3,0.5){\makebox(0,0){$G_{2}$}}
\end{picture}\caption{ $\ker(G_{1})=\left\{  x,y\right\}  \subset
\mathrm{core}(G_{1})=\left\{  x,y,v\right\}  $, while $\ker(G_{2}%
)=\mathrm{core}(G_{2})=\left\{  a,b\right\}  $.}%
\label{fig2}%
\end{figure}

\textbf{Data Availability Statement} On behalf of all authors, the
corresponding author states that the manuscript has no associated
data.\medskip

\textbf{Declarations\medskip}

\textbf{Conflict of Interest} On behalf of all authors, the corresponding
author states that there is no conflict of interest.

\end{document}